\documentclass[11pt,amssymb]{amsart}
\usepackage{amssymb}

\DeclareFontFamily{OT1}{rsfs}{}
\DeclareFontShape{OT1}{rsfs}{n}{it}{<-> rsfs10}{}
\DeclareMathAlphabet{\mathscr}{OT1}{rsfs}{n}{it}

\newcommand{\Z}{{\mathbb Z}}

\newcommand{\Q}{{\mathbb Q}}

\newcommand{\R}{{\mathbb R}}

\newcommand{\Br}{\mathrm{Br}}

\newcommand{\cL}{\mathscr{L}}

\newcommand{\md}{{\rm mod}\ }
\newcommand{\fX}{\mathfrak{X}}

\newcommand{\Ga}{\mathrm{Gal}}
\newtheorem{thm}{Theorem}[section]
\newtheorem{lemma}[thm]{Lemma}
\newtheorem{prop}[thm]{Proposition}
\newtheorem{cor}[thm]{Corollary}

\begin{document}
\title[Division algebras with the same maximal subfields]{On division algebras having the same maximal subfields}

\begin{abstract}
We show that if a field $K$ of characteristic $\neq 2$ satisfies the
following property (*) {\it for any two central quaternion division
algebras $D_1$ and $D_2$ over $K,$ the fact that $D_1$ and $D_2$
have the same maximal subfields implies that $D_1 \simeq D_2$ over
$K,$} then the field of rational functions $K(x)$ also satisfies
(*). This, in particular, provides an alternative proof for the
result of S.~Garibaldi and D.~Saltman that the fields of rational
functions $k(x_1, \ldots , x_r),$ where $k$ is a number field,
satisfy (*). We also show that $K = k(x_1, \ldots , x_r),$ where $k$
is either a totally complex number field with a single diadic place
(e.g. $k = \Q(\sqrt{-1})$) or a finite field of characteristic $\neq
2,$ satisfies the analog of (*) for all central division algebras
having exponent two in the Brauer group $\Br(K).$
\end{abstract}

\author[A.S.~Rapinchuk]{Andrei S. Rapinchuk$^\flat$}

\thanks{$^\flat$ Partially supported by NSF grant DMS-0502120  and the Humboldt Foundation.}

\address{Department of Mathematics, University of Virginia,
Charlottesville, VA 22904}

\email{asr3x@virginia.edu}

\author[I.A.~Rapinchuk]{Igor A. Rapinchuk}

\address{Department of Mathematics, Yale University, New Haven, CT 06502}

\email{igor.rapinchuk@yale.edu}

\maketitle

\section{Introduction}\label{S:I}

Given two (finite dimensional) central division algebras $D_1$ and
$D_2$ over the {\it same} field $K,$ we say that $D_1$ and $D_2$
have the {\it same maximal subfields} if any maximal subfield $F$ of
$D_1$ admits a $K$-embedding $F \hookrightarrow D_2,$ and vice
versa. In \cite{PR}, 5.4, a question was raised regarding fields $K$
having the following property:

\vskip2mm

(*) \parbox[t]{14cm}{\it For any two central {\rm quaternion}
division algebras $D_1$ and $D_2$ over $K,$ the fact that $D_1$ and
$D_2$ have the same maximal subfields implies that $D_1 \simeq D_2$
over $K.$}

\vskip2mm

\noindent This question was motivated by the analysis of the general
problem of when weak commensurability of Zariski-dense subgroups of
absolutely almost simple algebraic groups implies their
commensurability, which in turn is related to the well-known open
question as to whether or not two isospectral Riemann surfaces are
necessarily commensurable. The fact that number fields have (*) --
which follows from the Albert-Hasse-Brauer-Noether Theorem (AHBN)
(cf. the proof of Corollary \ref{C:Fin-1}) and is also a consequence
of the Minkowski-Hasse Theorem on quadratic forms --  was used in
\cite{PR} (cf. also \cite{Re}) to show that if one of the two
Zariski-dense subgroups in absolutely almost simple groups of type
$A_1$ is arithmetic and the two subgroups are weakly commensurable,
then they are actually commensurable (in particular, the other
subgroup is also arithmetic), which implies that if one of the two
isospectral Riemann surfaces is arithmetically defined then the
surfaces are commensurable. To extend this result to more general
Zariski-dense subgroups (first and foremost, to non-arithmetic
lattices in $\mathrm{SL}_2(\R)$) using the techniques developed in
\cite{PR}, one needs to know what other fields have (*). It was
observed by Rost, Wadsworth and others that it is possible to
construct ``large" (in particular, infinitely generated) fields
which do not have (*) (cf. \cite{GaSa}, Example 2.1). On the other
hand, no such examples are known for finitely generated fields (and
the fields that arise in the analysis of weakly commensurable
finitely generated Zariski-dense subgroups are finitely generated),
and the question as to what finitely generated fields have (*)
remains wide open. In fact, until recently no fields other than
global fields were known to have (*). In \cite{GaSa}, Garibaldi and
Saltman answered in the affirmative one of the questions posed in
earlier versions of \cite{PR} by showing that a purely
transcendental extension $K = k(x_1, \ldots , x_r)$ of a number
field $k$ has (*) (more generally, it was shown in \cite{GaSa} that
any {\it transparent} field of characteristic $\neq 2$ has (*)).

The goal of this note is to present two further results on (*) and
related issues. All fields below will be of characteristic $\neq 2.$
First, we show that the property to have (*) is stable under purely
transcendental extensions, which gives an alternative proof of the
fact that $K = k(x_1, \ldots , x_r),$ with $k$ a number field, has
(*).

\vskip2mm

\noindent {\bf Theorem A.} {\it Let $K$ be a field of characteristic
$\neq 2.$ If {\rm (*)} holds for $K$ then it also holds for the
field of rational functions $K(x).$}

\vskip2mm

Second, we give examples of fields for which a property similar to
(*) holds not only for quaternion, but for more general algebras
(so, the claim made in the end of \S 2 in \cite{GaSa} that (*)
cannot possibly hold for algebras of degree $> 2$ is not quite
accurate). We notice that given any central division algebra $D$
over $K,$ the opposite algebra $D^{\mathrm{op}}$ has the same
maximal subfields as $D$ (cf. Lemma \ref{L:B-301} for a more general
statement), so (the natural analog of) (*) definitely fails unless
$D \simeq D^{\mathrm{op}},$ i.e. the class $[D]$ has exponent two in
the Brauer group $\Br(K).$ On the other hand, there are division
algebras of exponent two that are more general than quaternion
algebras, and whether or not (*) holds for them is a meaningful
question (as the following theorem demonstrates).

\vskip2mm

\noindent {\bf Theorem B.} {\it Let $K = k(x_1, \ldots , x_r)$ be a
purely transcendental extension of a field $k$ which is either a
totally imaginary number field with a single diadic place (e.g. $k =
\Q(\sqrt{-1})$), or  is a finite field of characteristic $> 2,$ and
let $D_1$ and $D_2$ be two finite dimensional central division
algebras over $K$ such that the classes $[D_1] , [D_2] \in \Br(K)$
have exponent two. If $D_1$ and $D_2$ have the same maximal
subfields then $D_1 \simeq D_2.$}

\vskip2mm

One can view Theorem B as giving some indication that an analog of
(*) may hold over some fields more general than number fields for
some other absolutely almost simple algebraic groups associated with
algebras of exponent two. More precisely, following \cite{PR}, 5.4,
we call two $K$-forms $G_1$ and $G_2$ of an absolutely almost simple
algebraic group $G$ {\it weakly $K$-isomorphic} if they have the
same isomorphism classes of maximal $K$-tori, and the question is in
what situations weakly $K$-isomorphic groups are necessarily
$K$-isomorphic. Based on Theorem B, one can hope that the
affirmative answer is possible in certain cases where $G$ is of type
$B_n,$ $C_n$ and $G_2$ (and maybe $E_7$ and $F_4$) - see the end of
\S \ref{S:B} for a more detailed discussion, however none of these
types have been investigated so far.

The proofs of Theorems A (\S \ref{S:A}) and B (\S \ref{S:B}), just
as the argument in \cite{GaSa}, are based on an analysis of
ramification.

\vskip2mm

\noindent {\bf Notations and conventions.} All fields in this note
will be of characteristic $\neq 2.$ For a central simple algebra $A$
over a field $K,$  $[A]$ will denote the corresponding class in the
Brauer group $\Br(K).$ For $a , b \in K^{\times},$ we let
$\displaystyle \left(\frac{a , b}{K} \right)$ denote the
corresponding quaternion algebra. Given a valuation $v$ of a field
$K$ (and all valuations in this note will be discrete), we let
$\mathcal{O}_{K , v},$ $K_v$ and $\bar{K}^{(v)}$ denote the
corresponding valuation ring, the completion and the residue field,
respectively.

\vskip5mm

\section{Preliminaries}\label{S:P}

Let $K$ be a field endowed with a discrete valuation $v.$ For a
finite dimensional $K$-algebra $A,$ we set $A_v = A \otimes_K K_v$
endowed with the topology of a vector space over $K_v.$ We recall
that an \'etale $K$-algebra is defined to be a finite direct product
of finite separable extensions of $K.$ Then the notion for two
simple algebras to have the same maximal \'etale subalgebras is
defined in the obvious way (clearly, algebras with the same maximal
\'etale subalgebras have the same dimension).
\begin{lemma}\label{L:P-1}
{\rm (\cite{GaSa}, Lemma 3.1)} Let $A_1$ and $A_2$ be two central
simple algebras over $K,$ and let $v$ be a discrete valuation of
$K.$ If $A_1$ and $A_2$ have the same maximal \'etale subalgebras
then the algebras ${A_1}_{v}$ and ${A_2}_v$ also have the same
maximal \'etale subalgebras.
\end{lemma}
\begin{proof}
We give an argument based on a construction described in \cite{PR1},
proof of Theorem 3(ii) or in \cite{PR}, Lemma 3.4. Let $E_1$ be a
maximal \'etale subalgebra of ${A_1}_v.$ Let $G_1 = \mathrm{GL}_{1 ,
A_1}$ be the algebraic $K$-group associated to $A_1,$ and let $T_1 =
\mathrm{R}_{E_1/K_v}(\mathrm{GL}_1)$ be the maximal $K_v$-torus of
$G_1$ corresponding to $E_1.$ Consider the Zariski-open subset
$T_1^{\mathrm{reg}} \subset T_1$ of regular elements and the regular
map
$$
\varphi \colon G_1 \times T_1^{\mathrm{reg}} \longrightarrow G, \ \
(g , t) \mapsto gtg^{-1}.
$$
It is easy to check that the differential $d_{(g,t)}\varphi$ is
surjective for any $(g , t) \in G_1 \times T_1^{\mathrm{reg}},$ so
it follows from the Implicit Function Theorem that the map
$$
\varphi_v \colon G_1(K_v) \times T_1^{\mathrm{reg}}(K_v)
\longrightarrow G_1(K_v), \ \ (g , t) \mapsto gtg^{-1},
$$
is open in the topology defined by $v.$ In particular, $U_1 =
\mathrm{Im}\: \varphi_v$ is open in $G_1(K_v) = {A_1}_{v}^{\times}.$
On the other hand, by weak approximation for $K,$ we have that $A_1$
is dense in ${A_1}_v.$ So, pick $a \in A_1 \cap U_1$ and set $E_1^0
= K[a].$ Then $E_1^0$ is a maximal \'etale subalgebra of $A_1$ and
there exists $g \in {A_1}_v^{\times}$ such that $E_1^0 \otimes_K K_v
= gE_1g^{-1}.$ By our assumption, there exists an embedding $\iota^0
\colon E_1^0 \hookrightarrow A_2.$ Then
$$
\iota = (\iota^0 \otimes \mathrm{id}_{K_v}) \circ \mathrm{Int}\: g
$$
is a required embedding of $E_1$ into ${A_2}_v.$ By symmetry,
${A_1}_v$ and ${A_2}_v$ have the same maximal \'etale subalgebras.
\end{proof}

\vskip2mm

\addtocounter{thm}{1}

\noindent {\bf Remark 2.2.} The above argument uses the property of
weak approximation for $G_1$ which may fail for some algebraic
groups. Nevertheless, the following analog of Lemma \ref{L:P-1} is
true in general: Let $G_1$ and $G_2$ be two reductive algebraic
groups over a field $K,$ and let $v$ be a discrete valuation of $K.$
If $G_1$ and $G_2$ have the same isomorphism classes of maximal tori
over $K$ then they have the same isomorphism classes of maximal tori
over $K_v.$ To prove this, one needs to invoke weak approximation in
the variety of maximal tori - cf. the proof of Theorem 1(ii) in
\cite{PR2}.

\vskip2mm

\begin{lemma}\label{L:P-2}
Let $A_i = M_{d_i}(D_i)$ for some $d_i \geqslant 1$ and some central
division algebra $D_i$ over $K,$ where $i = 1, 2.$ If  $A_1$ and
$A_2$ have the same maximal \'etale subalgebras then $d_1 = d_2,$
and $D_1$ and $D_2$ have the same maximal subfields that are
separable over $K.$
\end{lemma}
\begin{proof}
Let $P_1$ be a maximal separable subfield of $D_1.$ Then $L_1 = P_1
\oplus \cdots \oplus P_1$ ($d_1$ times) is a maximal \'etale
subalgebra of $A_1.$ Let $e_1 = (1, 0, \ldots , 0),$ $\ldots ,$
$e_{d_1} = (0, \ldots , 0, 1)$ be the orthogonal idempotents in
$L_1.$ By our assumption, $L_1$ admits an embedding into $A_2,$ and
we identify the former with its image in the latter. Consider the
right $D_2$-vector space $V_2 = (D_2)^{d_2}$ as a left $A_2$-module.
Then the relations $e_i^2 = e_i$ and $e_ie_j = 0$ for $i \neq j$
imply that the sum of the nonzero $D_2$-subspaces $W_i = e_iV_2,$
where $i = 1, \ldots , d_1,$ is direct, hence $d_1 \leqslant d_2.$
By symmetry, $d_1 = d_2,$ and then $D_1$ and $D_2$ have the same
dimension. In the above notations, it follows that each $W_i$ is
1-dimensional over $D_2$ and $V_2 = \oplus W_i.$ Clearly, $P_1
\simeq L_1e_1$ embeds in $\mathrm{End}_{D_2}\: W_1 \simeq D_2,$ and
the required fact follows.
\end{proof}

\vskip2mm

\begin{cor}\label{C:P-1}
Let $A_1$ and $A_2$ be two central simple algebras over $K$ such
that $\dim_K A_1 = \dim_K A_2$ is relatively prime to
$\mathrm{char}\: K,$ and let $v$ be a discrete valuation of $K.$
Write ${A_i}_v = M_{d_i}(D_i)$ with $d_i \geqslant 1$ and $D_i$ a
central division $K_v$-algebra. If $A_1$ and $A_2$ have the same
maximal \'etale subalgebras then $d_1 = d_2$ and $D_1$ and $D_2$
have the same maximal subfields.
\end{cor}

Indeed, by Lemma \ref{L:P-1}, the $K_v$-algebras ${A_1}_v$ and
${A_2}_v$ have the same maximal \'etale subalgebras. So, by Lemma
\ref{L:P-2}, we have $d_1 = d_2$ and $D_1$ and $D_2$ have the same
maximal {\it separable} subfields. However, since $\dim_K A_1 =
\dim_K A_2$ is prime to $\mathrm{char}\: K,$ all maximal subfields
of $D_1$ and $D_2$ are separable over $K_v.$

\newpage

To formulate our next lemma, we need to introduce one condition on a
field $k:$

\vskip3mm

\noindent (LD)\ \  \parbox[t]{15cm}{Given finite separable
extensions $E \subset F$ of $k,$ any central division algebra
$\Delta$ over $E$ contains a maximal separable subfield $P$ that is
linearly disjoint from $F$ over $E.$}

\vskip3mm

\begin{lemma}\label{L:P-3}
Let $\mathcal{K}$ be a field complete with respect to a discrete
valuation $v,$ with  residue field $k = \bar{\mathcal{K}}^{(v)},$
and let $\mathcal{D}_1 , \mathcal{D}_2$ be two central division
algebras over $\mathcal{K}.$ Assume that the center $\mathcal{E}_i
:= Z(\bar{\mathcal{D}}_i)$ is a separable extension of $k$ for $i =
1 , 2.$ If $\mathcal{D}_1$ and $\mathcal{D}_2$ have the same maximal
subfields then

\vskip2mm

\noindent \ {\rm (i)} $[\mathcal{E}_1 : k] = [\mathcal{E}_2 : k];$

\vskip1mm

\noindent {\rm (ii)} \parbox[t]{14.5cm}{in each of the following
situations: (a) $\mathcal{D}_1$ and $\mathcal{D}_2$ are of prime
degree, (b) $k$ satisfies (LD), we have $\mathcal{E}_1 =
\mathcal{E}_2.$}
\end{lemma}
\begin{proof}
(i): Let $\tilde{v}_i$ be the extension of $v$ to $\mathcal{D}_i.$
It is known that $[\mathcal{E}_i : k]$ coincides with the
ramification index $e(\tilde{v}_i \vert v) =
[\tilde{v}_i(\mathcal{D}_i^{\times}) : v(\mathcal{K}^{\times})]$
(cf. \cite{Se},  ch. XII, \S 2, ex. 3, or \cite{W}, Thm. 3.4), so it
is enough to show that $e(\tilde{v}_1 \vert v) = e(\tilde{v}_2 \vert
v).$ Let $a \in \mathcal{D}_1$ be such that $\tilde{v}_1(a)$
generates the value group $\tilde{v}_1(\mathcal{D}_1^{\times}),$ and
let $\mathcal{L}$ be a maximal subfield of $\mathcal{D}_1$
containing $a.$ Then $e(\tilde{v}_1 \vert v) =
[\tilde{v}_1(\mathcal{L}^{\times}) : v(\mathcal{K}^{\times})].$ On
the other hand, by our assumption $\mathcal{L}$ can be embedded in
$\mathcal{D}_2,$ and if we identify the former with its image in the
latter, then $\tilde{v}_2 \vert \mathcal{L} = \tilde{v}_1 \vert
\mathcal{L}$ because $v$ has a unique extension to $\mathcal{L}$ as
$\mathcal{K}$ is complete (cf. \cite{Se}, ch. II, \S 2, cor. 2). It
follows that $e(\tilde{v}_1 \vert v) \leqslant e(\tilde{v}_2 \vert
v).$ By symmetry, $e(\tilde{v}_1 \vert v) = e(\tilde{v}_2 \vert v),$
as required.

\vskip2mm

(ii): First, suppose $\mathcal{D}_1$ and $\mathcal{D}_2$ have prime
degree $p.$ If $\mathcal{D}_1$ is unramified then $\mathcal{E}_1 =
k,$ so we obtain from (i) that $\mathcal{E}_2 = k,$ and there is
nothing to prove. If $\mathcal{D}_1$ is ramified then
$$
e(\tilde{v}_1 \vert v) = p = [\mathcal{E}_1 : k],
$$
so it follows from the formula in \cite{Se}, {\it loc. cit.}, that
$[\bar{\mathcal{D}}_1 : \mathcal{E}_1]$ divides $p,$ and therefore
in fact $\bar{\mathcal{D}}_1 = \mathcal{E}_1.$ Similarly,
$\bar{\mathcal{D}}_2 = \mathcal{E}_2.$ Pick $a$ in the valuation
ring $\mathcal{O}_{D_1 , \tilde{v}_1}$ so that for its residue
$\bar{a} \in \bar{\mathcal{D}}_1$ we have $\bar{\mathcal{D}}_1 =
k(\bar{a}).$ Then $\mathcal{L} := \mathcal{K}(a)$ is a maximal
unramified subfield of $\mathcal{D}_1,$ with residue field
$\bar{\mathcal{L}}^{(\tilde{v}_1)} = \mathcal{E}_1.$ As in the proof
of (i), $\mathcal{L}$ embeds into $\mathcal{D}_2$ and $\tilde{v}_1
\vert \mathcal{L} = \tilde{v}_2 \vert \mathcal{L}.$ So,
$[\bar{\mathcal{L}}^{(\tilde{v}_2)} : k] = p,$ hence
$\bar{\mathcal{L}}^{(\tilde{v}_2)} = \bar{\mathcal{D}}_2 =
\mathcal{E}_2.$ Finally,
$$
\mathcal{E}_1 = \bar{\mathcal{L}}^{(\tilde{v}_1)} =
\bar{\mathcal{L}}^{(\tilde{v}_2)} = \mathcal{E}_2,
$$
as required.

Now, assume that $k$ possesses property (LD), and let $n$ denote the
common degree of $\mathcal{D}_1$ and $\mathcal{D}_2.$ We will prove
that $\mathcal{E}_2 \subset \mathcal{E}_1;$ then  (i) will yield
$\mathcal{E}_1 = \mathcal{E}_2.$ Set $\mathcal{F} =
\mathcal{E}_1\mathcal{E}_2$ (in a fixed algebraic closure of $k$).
Since $\mathcal{E}_i$ is separable over $k,$ by (LD), there exists a
maximal separable subfield $\mathcal{P}$ of $\bar{\mathcal{D}}_1$
that is linearly disjoint from $\mathcal{F}$ over $\mathcal{E}_1.$
Write $\mathcal{P} = k(\bar{a}),$ and set $\mathcal{L} =
\mathcal{K}(a).$ Since $[\mathcal{P} : k] = n,$ we see that
$\mathcal{L}$ is a maximal subfield of $\mathcal{D}_1$ with residue
field $\bar{\mathcal{L}}^{(\tilde{v}_1)} = \mathcal{P}.$ As above,
$\mathcal{L}$ embeds into $\mathcal{D}_2,$ and therefore
$\mathcal{P} = \bar{\mathcal{L}}^{(\tilde{v}_1)} =
\bar{\mathcal{L}}^{(\tilde{v}_2)}$ embeds into
$\bar{\mathcal{D}}_2.$ It follows that $\mathcal{P} \supset
\mathcal{E}_2.$ (Indeed, otherwise $\mathcal{E}_2\mathcal{P}$ would
be a separable extension of $k$ contained in $\bar{D}_2$ and having
degree $> n.$ Writing $\mathcal{E}_2\mathcal{P} = k(\bar{b})$ for
some $b \in \mathcal{O}_{\mathcal{D}_2 , \tilde{v}_2},$ we would
find that $\mathcal{K}(b)$ would be an extension of $\mathcal{K}$
contained in $\mathcal{D}_2$ and of degree $> n,$ which is
impossible.) Since $\mathcal{P}$ was chosen to be linearly disjoint
from $\mathcal{F} = \mathcal{E}_1\mathcal{E}_2$ over
$\mathcal{E}_1,$ we conclude that $\mathcal{F} = \mathcal{E}_1,$
i.e. $\mathcal{E}_2 \subset \mathcal{E}_1.$
\end{proof}

\vskip2mm

\noindent {\bf Remark 2.6.} We would like to point out that the
assertion of Lemma \ref{L:P-3}(ii) that $\mathcal{E}_1 =
\mathcal{E}_2$ also holds if $\mathcal{D}_1$ and $\mathcal{D}_2$ are
of degree 4 and $k$ is of characteristic $\neq 2$ such that the
quaternion algebra $\displaystyle \left(\frac{-1 , -1}{k} \right)$
is not a division algebra (in particular, if $\sqrt{-1} \in k$).
Indeed, the argument in the cases where $[\mathcal{E}_1 : k] =
[\mathcal{E}_2 : k]$ equals 1 or 4 is identical to the one given in
Case (a) of Lemma \ref{L:P-3}(ii), so we only need to consider the
situation where $[\mathcal{E}_1 : k] = [\mathcal{E}_2 : k] = 2,$
hence $\bar{\mathcal{D}}_i$ is a central quaternion division algebra
over $\mathcal{E}_i$ for $i = 1, 2.$ To mimic the argument used in
Case (b), we need to show that $\bar{\mathcal{D}}_1$ contains a
maximal subfield $\mathcal{P}$ which is linearly disjoint from
$\mathcal{E}_1\mathcal{E}_2$ over $\mathcal{E}_1,$ and vice versa.
Since $[\mathcal{E}_1\mathcal{E}_2 : \mathcal{E}_1] \leqslant 2,$
this will obviously hold automatically if not all maximal subfields
of $\bar{\mathcal{D}}_1$ are $\mathcal{E}_1$-isomorphic. Now,
according to (\cite{Pi}, Ex. 4 in \S 13.6), if all maximal subfields
of a quaternion division algebra $D$ over a field $F$ of
characteristic $\neq 2$ are isomorphic then $F$ is formally real,
pythagorean and $\displaystyle D \simeq \left(\frac{-1 , -1}{F}
\right).$ However, our assumption implies that $\displaystyle
\left(\frac{-1 , -1}{\mathcal{E}_1} \right)$ is not a division
algebra, and the required fact follows.

\addtocounter{thm}{1}

\vskip2mm

We will now describe a class of fields having property (LD).
\begin{prop}\label{P:P-1}
Let $k$ be finitely generated over its prime subfield. Then $k$ has
property (LD).
\end{prop}
\begin{proof}
Let $E \subset F$ be finite separable extensions of $k,$ and let
$\Delta$ be a central division algebra over $E.$ Enlarging $F,$ we
can assume that $\Delta \otimes_E F \simeq M_n(F).$ It is enough to
construct a discrete valuation $w$ of $E$ so that the completion
$E_w$ is locally compact and coincides with $F_{\tilde{w}}$ for some
extension $\tilde{w} \vert w.$ Indeed, we can then pick a separable
extension $\mathcal{P}$ of $E_w$ of degree $n$ and embed it into
$\Delta_w = \Delta \otimes_E E_w \simeq M_n(E_w).$ The argument
given in the proof of Lemma \ref{L:P-1} (or the standard Krasner's
Lemma, cf. \cite{NSW}, Lemma 8.1.6) enables us to construct a
maximal subfield $P \subset \Delta$ such that $P \otimes_E E_w
\simeq \mathcal{P}.$ Then
$$
[PF : F] \geqslant [\mathcal{P}F_{\tilde{w}} : F_{\tilde{w}}] =
[\mathcal{P} : E_w] = n.
$$
So, $[PF : F] = n,$ and therefore $P$ and $F$ are linearly disjoint
over $E.$

In characteristic zero, according to Proposition 1 in \cite{PR1},
for infinitely many primes $p$ there exists an embedding $F
\hookrightarrow \Q_p,$ and then the valuations $w$ and $\tilde{w}$
of $E$ and $F$ respectively, obtained as pull-backs of the $p$-adic
valuation, are as required. If $\mathrm{char}\: k = p > 0$ then we
need to use a suitable modification of the proof of Proposition 1 in
\cite{PR1}. Let $\mathbb{F}_p$ be the field with $p$ elements. There
exist algebraically independent $t_1, \ldots , t_r \in E$ such that
$F$ is a finite separable extension of $\ell :=  \mathbb{F}_p(t_1,
\ldots , t_r).$ Furthermore, we can pick a primitive element $a \in
F$ over $\ell$ having minimal polynomial $f$ of the form
$$
f(s, t_1, \ldots , t_r) = s^d + c_{d-1}(t_1, \ldots , t_r)s^{d-1} +
\cdots + c_0(t_1, \ldots , t_r)
$$
where $c_i(t_1, \ldots , t_r) \in \mathbb{F}_p[t_1, \ldots , t_r].$
Since $f$ is prime to its derivative $f_s,$ there exist polynomials
$g(s, t_1, \ldots , t_r),$ $h(s, t_1, \ldots , t_r)$ and $m(t_1,
\ldots , t_r)$ with coefficients in $\mathbb{F}_p$ such that
\begin{equation}\label{E:P101}
gf + hf_s = m \neq 0.
\end{equation}
One can find polynomials $\alpha_1(t), \ldots , \alpha_r(t) \in
\mathbb{F}_p[t]$ such that $\alpha_1(t) = t$ and $m(\alpha_1(t),
\ldots , \alpha_r(t)) \neq 0.$ Set $\varphi(s) = f(s, \alpha_1(t),
\ldots , \alpha_r(t)).$
%
%
%Besides, using L\"uroth's Theorem, we can change $t$ to a different
%parameter to ensure that $\mathbb{F}_p(\alpha_1(t), \ldots ,
%\alpha_r(t)) = \mathbb{F}_p(t).$ Then we can write
%\begin{equation}\label{E:P102}
%t = \frac{G(\alpha_1(t), \ldots , \alpha_r(t))}{H(\alpha_1(t),
%\ldots , \alpha_r(t))},
%\end{equation}
%where $G$ and $H$ are some polynomials with coefficients in
%$\mathbb{F}_p.$
Let $\beta$ be a root of $\varphi(s)$ (in a fixed algebraic closure
of $\mathbb{F}_p(t)$), and let $R = \mathbb{F}_p(t)(\beta).$ By
Tchebotarev's Density Theorem (cf. \cite{ANT}, Ch. VII, 2.4), one
can find a valuation $v$ of $\mathbb{F}_p(t)$ associated to some
irreducible polynomial in $\mathbb{F}_p[t]$ such that
\begin{equation}\label{E:P103}
v(m(\alpha_1(t), \ldots , \alpha_r(t))) = 0,
\end{equation}
and $R_{\tilde{v}} = \mathbb{F}_p(t)_v$ for some extension
$\tilde{v} \vert v.$ Let $\mathbb{F}_q = \mathcal{O}_{R ,
\tilde{v}}/\mathfrak{P}_{\tilde{v}} = \mathcal{O}_{\mathbb{F}_p(t) ,
v}/\mathfrak{p}_v$ be the residue field, and let $\beta^0,
\alpha_1^0, \ldots , \alpha_r^0$ be the images of $\beta,
\alpha_1(t), \ldots , \alpha_r(t)$ in $\mathbb{F}_q.$ It follows
from (\ref{E:P101}) and (\ref{E:P103}) that
\begin{equation}\label{E:P103.5}
f(\beta^0, \alpha_1^0, \ldots , \alpha_r^0) = 0, \ \ \text{but} \ \
f_s(\beta^0, \alpha_1^0, \ldots , \alpha_r^0) \neq 0.
\end{equation}
%and we also see from (\ref{E:P102}) and (\ref{E:P103}) that
%\begin{equation}\label{E:P104}
%\mathbb{F}_q = \mathbb{F}_p(\alpha_1^0, \ldots , \alpha_r^0).
%\end{equation}
Let $\mathcal{F} = \mathbb{F}_q((T)).$ Pick $r$ algebraically
independent over $\mathbb{F}_q$ elements $\tilde{t}_1, \ldots ,
\tilde{t}_r \in \mathbb{F}_q[[T]]$ of the form:
\begin{equation}\label{E:P105}
\tilde{t}_1 = \alpha_1^0 + T, \ \ \tilde{t}_i = \alpha_i^0(\md T) \
\ \text{for} \ i > 1,
\end{equation}
and consider the embedding $\ell \hookrightarrow \mathcal{F}$
sending $t_i$ to $\tilde{t}_i$ for all $i = 1, \ldots , r.$ We claim
that
\begin{equation}\label{E:P106}
\overline{\iota(\ell)} = \mathcal{F}.
\end{equation}
Indeed, it follows from (\ref{E:P105}) that the image of
$\iota(t_1)$  in $\mathbb{F}_q[[T]]/T\mathbb{F}_q[[T]] =
\mathbb{F}_q$ generates $\mathbb{F}_q,$ which implies (e.g. by
Hensel's Lemma) that $\mathbb{F}_q \subset \overline{\iota(\ell)}.$
We then see from (\ref{E:P105}) that $T \in \overline{\iota(\ell)},$
and (\ref{E:P106}) follows. To complete the argument, we will now
that that $\iota$ can be extended to an embedding $\tilde{\iota}
\colon F \hookrightarrow \mathcal{F}$ as then the pullbacks of the
natural valuation on $\mathcal{F}$ will give us the required
valuations $w$ and $\tilde{w}$ on $E$ and $F$ respectively. Using
Hensel's Lemma, one obtains from (\ref{E:P103.5}) and (\ref{E:P105})
that $\mathcal{F}$ contains a root to $f(s, \tilde{t}_1, \ldots
\tilde{t}_r) = 0,$ and the existence of $\tilde{\iota}$ follows.
\end{proof}

\section{Proof of Theorem B}\label{S:B}

First, we will first single out some conditions on the field $K$
that imply the assertion of Theorem~B (cf. Theorem \ref{T:B-1}). We
will then verify these conditions for the fields considered in
Theorem B (cf. Proposition \ref{P:B-1}).

\vskip1mm

For a field $\mathcal{K}$  complete with respect to a discrete
valuation $v,$ we let $\Br^0(\mathcal{K}),$ or
$\Br^0_v(\mathcal{K}),$ denote the subgroup of $\Br(\mathcal{K})$
consisting of elements that split over an unramified extension of
$\mathcal{K}$ (in other words, $\Br^0(\mathcal{K}) =
\Br(\mathcal{K}_{\small{\mathrm{nr}}}/\mathcal{K})$ where
$\mathcal{K}_{\small{\mathrm{nr}}}$ is the maximal unramified
extension of $\mathcal{K}$)\footnote{As usual, the definition of an
unramified extension $\mathcal{L}/\mathcal{K}$ includes the
requirement that the corresponding extension of the residue fields
$\bar{\mathcal{L}}^{(v)}/\bar{\mathcal{K}}^{(v)}$ be separable.}. We
recall that for a central division algebra $\mathcal{D}$ over
$\mathcal{K},$ we have $[\mathcal{D}] \in \Br^0(\mathcal{K})$ if and
only if the center $Z(\bar{\mathcal{D}}^{(v)})$ of the residue
algebra is a separable extension of $\bar{\mathcal{K}}^{(v)}$ (cf.
\cite{W}, Theorem 3.4); the elements of $\Br^0(\mathcal{K})$ are
called ``inertially split" (cf. \cite{Sch}, \cite{W}). It follows
that if $n$ is relatively prime to $\mathrm{char}\:
\bar{\mathcal{K}}^{(v)}$ then the $n$-torsion subgroup
$\Br(\mathcal{K})_n$ is contained in $\Br^0(\mathcal{K}).$

Furthermore, we let $\rho$ or $\rho_v$ denote the reduction map
$\Br^0_v(\mathcal{K}) \to \mathrm{Hom}(G(v) , \Q/\Z)$ where $G(v)$
is the absolute Galois group of the residue field
$\bar{\mathcal{K}}^{(v)}$ (cf., for example, \cite{Se}, ch. XII, \S
3, or \cite{W}, (3.9)). Then $\Br'_v(\mathcal{K}) := \mathrm{Ker}\:
\rho_v$ is known to consist precisely of the classes of all
unramified (or ``inertial") division algebras (equivalently, those
division algebras that arise from the Azumaya algebras over the
valuation ring $\mathcal{O}_{\mathcal{K} , v},$ cf. \cite{W},
Theorem 3.2). More generally, given a discrete valuation $v$ of a
field $K,$ we let $\Br'_v(K)$ denote the subgroup of classes $[A]
\in \Br(K)$ for which $[A \otimes_K K_v] \in \Br'_v(K_v).$

\vskip3mm

\noindent {\bf Definition 3.1.} Let $K$ be an infinite field of
characteristic $\neq 2.$  We say that $K$ is {\it 2-balanced} if
there exists a set $V$ of discrete valuations of $K$ such that

\vskip2mm

(a) \parbox[t]{15cm}{for each $v \in V,$ the residue field
$\bar{K}^{(v)}$ satisfies (LD) and is of characteristic $\neq 2;$}

\vskip1mm

(b) \parbox[t]{15cm}{$\displaystyle \bigcap_{v \in V} \Br'_v(K)_2 =
\{ e \}$ (in other words, the 2-component of the unramified Brauer
group of $K$ with respect to $V$ is trivial).}

\addtocounter{thm}{1}

\vskip3mm

\begin{thm}\label{T:B-1}
Let $K$ be a 2-balanced field, and let $D_1 , D_2$ be two central
division algebras over $K$ such that $[D_1] , [D_2] \in \Br(K)_2.$
If $D_1$ and $D_2$ have the same maximal subfields then $D_1 \simeq
D_2.$
\end{thm}
\begin{proof}
For $v \in V,$ we let $\rho_v$ denote the reduction map $\Br^0(K_v)
\to \mathrm{Hom}(G(v) , \Q/\Z).$ The assumption that
$\mathrm{char}\: \bar{K}^{(v)} \neq 2$ implies that $[D_1 \otimes_K
K_v] , [D_2 \otimes_K K_v] \in \Br^0(K_v),$ and then due to
condition (b) in the above definition, it is enough to show that
\begin{equation}\label{E:B-200}
\rho_v([D_1 \otimes_K K_v]) = \rho_v([D_2 \otimes_K K_v]),
\end{equation}
for all $v \in V.$ Fix $v \in V,$ and set $\mathcal{K} = K_v.$
According to Lemma \ref{L:P-1}, the algebras $D_1 \otimes_K
\mathcal{K}$ and $D_2 \otimes_K \mathcal{K}$ have the same maximal
\'etale subalgebras. Using Corollary \ref{C:P-1}, we see that
$$
D_i \otimes_K \mathcal{K} = M_{d_i}(\mathcal{D}_i), \ \ i = 1, 2,
$$
where $d_1 = d_2$ and the division algebras $\mathcal{D}_1$ and
$\mathcal{D}_2$ have the same maximal subfields. To prove
(\ref{E:B-200}), it suffices to show that
\begin{equation}\label{E:B-201}
\rho_v([\mathcal{D}_1]) = \rho_v([\mathcal{D}_2]).
\end{equation}
Let $\chi_i = \rho_v([\mathcal{D}_i]) \in \mathrm{Hom}(G(v) ,
\Q/\Z).$ Since $\mathcal{D}_1$ and $\mathcal{D}_2$ have the same
maximal subfields and $\bar{K}^{(v)} = \bar{\mathcal{K}}^{(v)}$
satisfies (LD), we conclude from Lemma \ref{L:P-3}(ii), case (b),
that $Z(\bar{\mathcal{D}}_1) = Z(\bar{\mathcal{D}}_2).$ It is known,
however, that the extension
$Z(\bar{\mathcal{D}}_i)/\bar{\mathcal{K}}$ corresponds
$\mathrm{Ker}\: \chi_i$ (cf. \cite{W}, Theorem 3.5). So, we obtain
that $\mathrm{Ker}\: \chi_1 = \mathrm{Ker}\: \chi_2,$ and since
$\chi_1$ and $\chi_2$ both have order two, we conclude that $\chi_1
= \chi_2,$ yielding (\ref{E:B-201}).
\end{proof}

\vskip2mm

\noindent {\bf Remark 3.3.} Using Remark 2.6 in place of Lemma
\ref{L:P-3}(ii), we obtain that if a field $K$ with the property
that the quaternion algebra $\displaystyle \left(\frac{-1 , -1}{K}
\right)$ is not a division algebra, has a set $V$ of discrete
valuations such that $\mathrm{char}\: \bar{K}^{(v)} \neq 2$ for each
$v \in V$ and $\bigcap_{v \in V} \Br'_v(K) = \{ e \}$ then for any
central division algebras $D_1$ and $D_2$ of degree four over $K$
such that $[D_1] , [D_2] \in \Br(K)_2,$ the fact that $D_1$ and
$D_2$ have the same maximal subfields implies that $D_1 \simeq D_2$
(we only need to observe that $\displaystyle \left(\frac{-1 ,
-1}{\bar{K}^{(v)}} \right)$ is not a division algebra for all $v \in
V$).

\addtocounter{thm}{1}

\vskip2mm

The following proposition establishes that the fields in the
statement of Theorem B are 2-balanced, completing thereby the proof
of the latter.
\begin{prop}\label{P:B-1}
Let $K = k(x_1, \ldots , x_r)$ be a purely transcendental extension
of a field $k$ which is either a totally imaginary number field with
a single diadic place (e.g. $k = \Q(\sqrt{-1})$) or a finite field
of characteristic $> 2.$ Then $K$  is 2-balanced.
\end{prop}
\begin{proof}
Let $K_i = k(x_1, \ldots , x_{i-1} , x_{i+1} , \ldots x_r),$ and let
$V_i$ be the set of all valuations of $K = K_i(x_i)$ that are
trivial on $K_i.$ Then for $v \in V_i$ the residue field
$\bar{K}^{(v)}$ is a finite extension of $K_i,$ hence a finitely
generated field of characteristic $\neq 2.$ Invoking Proposition
\ref{P:P-1}, we see that $\bar{K}^{(v)}$ satisfies (LD), and
therefore
$$
V_0 := \bigcup_{i = 1}^r V_i
$$
satisfies condition (a) of Definition 3.1. Henceforth, we will
identify $\Br(k)$ with a subgroup of $\Br(K)$ using the natural
embedding. We claim that
\begin{equation}\label{E:B-305}
\bigcap_{v \in V_0} \Br'_v(K)_2 = \Br(k)_2.
\end{equation}
This is proved by induction on $r$ using the following consequence
of Faddeev's exact sequence (cf. \cite{GS}, Cor. 6.4.6, or
\cite{Pi}, \S 19.5; for the case of a nonperfect field of constants,
see \cite{GMS}, Example 9.21 on p. 26, and \cite{RoSiTi}): Let $F$
be a field of characteristic $\neq 2,$ and let $V^F$ be the set of
valuations of the field of rational functions $F(x)$ that are
trivial on $F;$ then
\begin{equation}\label{E:B-306}
\bigcap_{v \in V^F} \Br'_v(F(x))_2 = \Br(F)_2.
\end{equation}
For $r = 1,$ (\ref{E:B-305}) is identical to (\ref{E:B-306}), and
there is nothing to prove. For $r > 1,$ set $k' = k(x_r)$ and $V'_0
= \bigcup_{i = 1}^{r-1} V_i.$ By induction hypothesis,
\begin{equation}\label{E:B-307}
\bigcap_{v \in V'_0} \Br'_v(K)_2 = \Br(k')_2.
\end{equation}
On the other hand, there is a natural bijection between $V^{k'}$ and
$V_r,$ $v \mapsto \hat{v}.$ Clearly, if a central division algebra
$\Delta$ over $k'$ (of exponent two) is ramified at $v \in V^{k'},$
i.e. $[\Delta] \notin \Br'_v(k')_2,$ then $\hat{\Delta} = \Delta
\otimes_{k'} K$ is ramified at $\hat{v},$ i.e. $[\hat{\Delta}]
\notin \Br'_{\hat{v}}(K).$ Thus, if a central division algebra $D$
over $K$ represents an element from the left-hand side of
(\ref{E:B-305}), then by (\ref{E:B-307}) we can write $D = \Delta
\otimes_{k'} K$ for some central division algebra $\Delta$ over
$k'.$ Furthemore, it follows from our previous remark that $[\Delta]
\in \bigcap_{v \in V^{k'}} \Br'_v(k').$ So, using (\ref{E:B-306})
for $F = k,$ we see that $[\Delta] \in \Br(k),$ proving
(\ref{E:B-305}). It immediately follows that $V = V_0$ is as
required if $k$ is finite.

Let now $k$ be a totally imaginary number field with a single diadic
place. It follows from the Albert-Hasse-Brauer-Noether Theorem (cf.
\cite{Pi}, \S 18.4) that for the set $W$ of all non-diadic
nonarchimedean places of $k$ the natural map
$$
\Br(k) \longrightarrow \bigoplus_{w \in W} \Br(k_w)
$$
is injective. Since $\Br'_w(k_w) = \{ e \},$ this can be restated as
\begin{equation}\label{E:B-310}
\bigcap_{w \in W} \Br'_w(k) = \{ e \}.
\end{equation}
For $w \in W,$ we let $\tilde{w}$ denote its natural extension to
$K$ given by
$$
\tilde{w}\left( \sum_{i_1, \ldots , i_r} a_{i_1 \ldots i_r}
x_1^{i_1} \cdots x_r^{i_r} \right) = \inf w(a_{i_1 \ldots i_r}).
$$
Then the residue field $\bar{K}^{(\tilde{w})} = \bar{k}^{(w)}(x_1,
\ldots , x_r)$ is a finitely generated field of characteristic $\neq
2,$ hence satisfies (LD) (cf. Proposition \ref{P:P-1}). It follows
that $V := V_0 \cup \tilde{W},$ where $\tilde{W} = \{ \tilde{w} \:
\vert \: w \in W \},$ satisfies condition (a) of Definition 3.1. At
the same time, using (\ref{E:B-305}) and (\ref{E:B-310}) and arguing
as above, we see that
$$
\bigcap_{v \in V} \Br'_v(K)_2 = \{ e \},
$$
which is condition (b) of Definition 3.1. Thus, $V$ is as required.
\end{proof}

\vskip2mm

As we already mentioned in the introduction, for any central
division $K$-algebra $D,$ the opposite algebra $D^{\mathrm{op}}$ has
the same maximal subfields, but $D \not\simeq D^{\mathrm{op}}$
unless $[D] \in \Br(K)_2.$ It should be noted, however, that the
associated norm one groups groups $\mathrm{SL}_{1 , D}$ and
$\mathrm{SL}_{1 , D^{\mathrm{op}}}$ are always $K$-isomorphic. So,
we would like to point out the following general construction of
division algebras $D_1$ and $D_2$ that have the same maximal
subfields, but for which $\mathrm{SL}_{1 , D_1} \not\simeq
\mathrm{SL}_{1 , D_2}.$ Let $\Delta_1$ and $\Delta_2$ be two central
division algebras over $K$ of relatively prime degrees $n_1 , n_2 >
2.$ Then $D_1 = \Delta_1 \otimes_K \Delta_2$ and $D_2 = \Delta_1
\otimes_K \Delta_2^{\mathrm{op}}$ are division algebras of degree $n
= n_1n_2,$ which are neither isomorphic nor anti-isomorphic, so the
corresponding norm one groups are not $K$-isomorphic. At the same
time, if $P$ is a maximal subfield of $D_1$ then it splits
$\Delta_1$ and $\Delta_2,$ hence also $\Delta_2^{\mathrm{op}}.$ It
follows that $P$ splits $D_2,$ and therefore is isomorphic to its
maximal subfield. A more general perspective on this construction
can be derived from the following (known) statement (to see the
connection, one observes that for the algebras $D_1$ and $D_2$ as
above, the classes $[D_1] , [D_2] \in \Br(K)$ generate the same
subgroup, which coincides with the subgroup generated by
$[\Delta_1]$ and $[\Delta_2],$ hence $[D_2] = m[D_1]$ for some $m$
relatively prime to $n$).
\begin{lemma}\label{L:B-301}
Let $D$ be a central division algebra of degree $n$ over a field
$K.$ Then for any $m \geqslant 1$ which is relatively prime to $n,$
the class $m[D] \in \Br(K)$ is represented by a central division
algebra $D_m$ of the same degree $n$ and having the same maximal
subfields as $D.$
\end{lemma}
\begin{proof}
First, we observe that if $A_1$ and $A_2$ are two central simple
algebras over $K$ of the same degree $d$ containing a field
extension $P/K$ of degree $d$ then $A := A_1 \otimes_K A_2$ is
isomorphic to $M_d(A')$ where $A'$ is a central simple algebra of
degree $d$ that also contains $P.$ Indeed, we have $A_1 \otimes_K P
\simeq M_d(P)$ so $A$ contains $B := M_d(K).$ Using the Double
Centralizer Theorem, we conclude that $A \simeq B \otimes_K C_A(B),$
i.e. $A \simeq M_d(A')$ where $A' = C_A(B)$ is a central simple
algebra of degree $d.$ Since
$$
A \otimes_K P \simeq (A_1 \otimes_K P) \otimes_P (A_2 \otimes_K P)
\simeq M_{d^2}(P),
$$
we see that $P$ splits $A'$ and therefore is isomorphic to a maximal
subfield of the latter (cf. \cite{Pi}, \S 13.3). This remark
combined with  simple induction shows that for any $m \geqslant 1$
we have
\begin{equation}\label{E:B-302}
D^{\otimes m} \simeq M_{n^{m-1}}(D_m)
\end{equation}
where $D_m$ is a central simple algebra of degree $n$ such that
every maximal subfield of $D$ embeds in $D_m.$ Let now $m$ be
relatively prime to $n,$ and pick $\ell \geqslant 1$ so that $\ell m
\equiv 1(\md n),$ hence $\ell[D_m] = [D].$ If $D_m = M_s(\Delta)$
where $\Delta$ is a division algebra of degree $t,$ $st=n,$ then
$\ell [D_m] = \ell [\Delta].$ Applying (\ref{E:B-302}) to $\Delta$
and $\ell$ in place of $D$ and $m,$ we see that $\ell [\Delta]$ is
represented by a central simple algebra $\Delta_{\ell}$ of degree
$t.$ Then $[\Delta_{\ell}] = [D]$ is possible only if $D_m$ is
isomorphic to $\Delta,$ hence a division algebra, and $\Delta_{\ell}
\simeq D.$ Furthermore, as we have seen, every maximal subfield of
$D$ embeds in $D_m,$ and every maximal subfield of $D_m \simeq
\Delta$ embeds in $\Delta_{\ell} \simeq D,$ i.e. $D$ and $D_m$ have
the same maximal subfields.
\end{proof}

\vskip0.5mm

Lemma \ref{L:B-301} seems to suggest that as a potential
generalization of (*) to arbitrary division algebras one should
consider the question of whether for two central division algebras
$D_1$ and $D_2$ over a field $K,$ the classes $[D_1]$ and $[D_2]$
generate the same subgroup of $\Br(K).$ Unfortunately, the answer to
this question is negative already over number fields for algebras of
any degree (exponent) $n > 2.$ To see this, one can pick four
nonarchimedean places $v_1, \ldots , v_4$ of a given number $K,$ and
then consider, for any $n > 2,$ the division algebras $D_1$ and
$D_2$ over $K$ having local invariants $1/n, 1/n, -1/n, -1/n$ and
$1/n, -1/n, 1/n, -1/n$ respectively at $v_1, \ldots , v_4,$ and $0$
everywhere else (cf. \cite{PR}, Example 6.5). It follows from
(\cite{Pi}, Cor.~b in \S 18.4) that $D_1$ and $D_2$ have the same
maximal subfields; at the same time, $[D_1]$ and $[D_2]$ generate
different subgroups of $\Br(K).$ Thus, there appears to be no
sensible analog of Theorem B for algebras of exponent $> 2.$ On the
other hand, in addition to generalizing Theorem B to other fields,
one may consider similar questions for algebras with involution.
More precisely, let $(A_1 , \tau_1)$ and $(A_2 , \tau_2)$ be two
central simple algebras over $K$ with involutions of the first kind
(i.e., acting trivially on $K$) and of the same type (symplectic or
orthogonal) - then of course $[A_1] , [A_2] \in \Br(K)_2.$ Assume
that $A_1$ and $A_2$ have the same isomorphism classes of maximal
\'etale subalgebras invariant under the involutions\footnote{This
assumption has two possible interpretations: for any
$\tau_1$-invariant maximal \'etale subalgebra $E_1 \subset A_1$
there exist a $\tau_2$-invariant maximal \'etale subalgebra $E_2
\subset A_2$ such that $E_1 \simeq E_2$ as $K$-algebras, or such
that $(E_1 , \tau_1 \vert E_1) \simeq (E_2 , \tau_2 \vert E_2)$ as
$K$-algebras with involutions (and vice versa).}. In what situations
can one guarantee that $A_1 \simeq A_2$ as $K$-algebras? $(A_1 ,
\tau_1) \simeq (A_2 , \tau_2)$ as $K$-algebras with involutions?
(Affirmative) results in this direction may lead to some progress on
the problem, mentioned in the introduction, of when two weakly
isomorphic forms of an absolutely almost simple algebraic group are
necessarily isomorphic, particularly for types $B_n$ and $C_n.$

\vskip.5cm

\section{Proof of Theorem A}\label{S:A}

For a quaternion algebra $\displaystyle D = \left(\frac{a ,
b}{K}\right)$ corresponding to $a , b \in K^{\times}$ we let $q_D$
denote the quadratic form
$$
q_D(s, t, u) = as^2 + bt^2 - abu^2.
$$
Then $D$ is a division algebra if and only if $q_D$ does not
represent nonzero squares, in which case the maximal subfields of
$D$ are isomorphic to the quadratic extensions of the form
$K(\sqrt{d})$ where $d \in K^{\times}$ is represented by $q_D.$
Furthermore, it is known (\cite{Pi}, \S 1.7) that two central
quaternion algebras $D_1$ and $D_2$ over $K$ are isomorphic if and
only if the corresponding quadratic forms $q_{D_1}$ and $q_{D_2}$
are equivalent. Thus, (*) reduces to the statement that the
quadratic forms $q_{D_1}$ and $q_{D_2}$ associated to two quaternion
division algebras $D_1$ and $D_2$ are equivalent given that they
represent the same elements of $K.$

\vskip1mm

The proof of Theorem A is based on Faddeev's exact sequence (cf.
\cite{GS}, Cor. 6.4.6,  \cite{Pi}, \S 19.5, \cite{GMS}, Example 9.21
on p. 26, or \cite{RoSiTi}). Let $K^{\mathrm{sep}}$ be a separable
closure of $K,$ and $G = \Ga(K^{\mathrm{sep}}/K)$ be its absolute
Galois group. Furthermore, let $V$ be the set of valuations of
$K(x)$ corresponding to all irreducible polynomials $p(x) \in K[x].$
For each $v \in V,$ we fix its extension $\tilde{v}$ to
$K^{\mathrm{sep}}(x),$ and let $G(v) = G(\tilde{v} \vert v)$ be the
corresponding decomposition group; we observe that $G(v)$ is
naturally identified with the absolute Galois group of the residue
field $\overline{K(x)}^{(v)}.$ Then we have the following exact
sequence:
\begin{equation}\label{E:A-1}
0 \to \Br(K) \stackrel{\iota}{\longrightarrow}
\Br(K^{\mathrm{sep}}(x)/K(x)) \stackrel{\phi}{\longrightarrow}
\bigoplus_{v \in V} \mathrm{Hom}(G(v) , \Q/\Z).
\end{equation}
in which $\iota$ is the natural embedding $[A] \mapsto [A \otimes_K
K(x)],$ and $\phi = (\phi_v),$ where the local components $\phi_v$
are related to the reduction maps $\rho_v \colon \Br^0_v(K(x)_v) \to
\mathrm{Hom}(G(v) , \Q/\Z)$ by $\phi_v = \rho_v \circ \lambda_v$
with $\lambda_v \colon \Br(K(x)) \to \Br(K(x)_v)$ being the natural
map (notice that $\lambda_v(\Br(K^{\mathrm{sep}}(x)/K(x))) \subset
\Br^0_v(K(x)_v)$). Since $\mathrm{char}\: K \neq 2,$ we have the
inclusion $\Br(K(x))_2 \subset \Br(K^{\mathrm{sep}}(x)/K(x)),$ so
(\ref{E:A-1}) gives rise to the following exact sequence
\begin{equation}\label{E:A-2}
0 \to \Br(K)_2 \stackrel{\iota}{\longrightarrow} \Br(K(x))_2
\stackrel{\phi}{\longrightarrow} \bigoplus_{v \in V}
\mathrm{Hom}(G(v) , \mu_2) \ \ \text{where} \ \ \mu_2 = \{ \pm 1 \}.
\end{equation}
\begin{lemma}\label{L:1}
{\rm (Cf. \cite{GaSa}, Proposition 3.5)} Let $D_1$ and $D_2$ be two
central quaternion division algebras over $K(x)$ having the same
maximal subfields. Then $\phi([D_1]) = \phi([D_2]).$
\end{lemma}
\begin{proof} (Cf. the proof of Theorem \ref{T:B-1}) Fix $v \in V.$
By Corollary \ref{C:P-1}, we have
$$
D_i \otimes_{K(x)} K(x)_v = M_{d_i}(\mathcal{D}_i), \ \ i = 1, 2,
$$
where $d_1 = d_2$ and the division algebras $\mathcal{D}_1$ and
$\mathcal{D}_2$ have the same maximal subfields. Using
Lemma~\ref{L:P-3}(ii), case (a), we see that $Z(\bar{\mathcal{D}}_1)
= Z(\bar{\mathcal{D}}_2).$ Since the extension
$Z(\mathcal{D}_i)/\overline{K(x)}^{(v)}$ corresponds to the subgroup
$\mathrm{Ker}\: \chi_i \subset G(v),$ where $\chi_i =
\rho_v([\mathcal{D}_i]),$ we conclude that $\mathrm{Ker}\: \chi_i =
\mathrm{Ker}\: \chi_2,$ and eventually $\chi_1 = \chi_2,$ implying
that $\phi_v([D_1]) = \phi_v([D_2]).$
\end{proof}

\vskip.5mm

\begin{cor}\label{C:1}
Let $D_1$ and $D_2$ be as in Lemma \ref{L:1}. Then there exist a
central quaternion algebra $D$ over $K$ such that $D_1
\otimes_{K(x)} D_2 \simeq M_2(D \otimes_K K(x)).$
\end{cor}
\begin{proof}
It follows from Lemma \ref{L:1} that there exists a division algebra
$\mathcal{D}$ over $K$ such that $$[D_1 \otimes_{K(x)} D_2] =
[\mathcal{D} \otimes_K K(x)].$$ Notice that $\mathcal{D} \otimes_K
K(x)$ is also a division algebra. On the other hand, since $D_1$ and
$D_2$ possess a common subfield, we have $D_1 \otimes_{K(x)} D_2
\simeq M_2(\Delta)$ for some central quaternion algebra $\Delta$
over $K(x)$ (cf. \cite{GS}, Lemma 1.5.2). From the uniqueness in
Wedderburn's Theorem we conclude that either $\mathcal{D} = K$ or
$\mathcal{D}$ is a quaternion algebra over $K.$ Set $D = M_2(K)$ in
the former case, and $D = \mathcal{D}$ in the latter. Then our
construction implies that $D$ is as required.
\end{proof}

\vskip.5mm

Now, to complete the proof of Theorem A, it remains to show that in
the notations of Corollary~\ref{C:1}, the class $[D]$ in $\Br(K)$ is
trivial. For this, we need to make a couple of preliminary
observations.
\begin{lemma}\label{L:A-201}
Let $\mathcal{K}$ be a field with a discrete valuation $v$ such that
$\mathrm{char}\: \bar{\mathcal{K}}^{(v)} \neq 2.$ For $i = 1, 2,$
let $a_i , b_i \in \mathcal{K}^{\times}$ be such that $v(a_i) =
v(b_i) = 0,$ and set
$$
\mathcal{D}_i = \left(\frac{a_i , b_i}{\mathcal{K}}  \right) \ \ , \
\ \check{\mathcal{D}}_i = \left(\frac{\bar{a}_i ,
\bar{b}_i}{\bar{\mathcal{K}}^{(v)}} \right),
$$
where $\bar{a}_i , \bar{b}_i$ are the images of $a_i , b_i$ in
$\bar{\mathcal{K}}^{(v)}.$ Assume that $\mathcal{D}_1$ and
$\mathcal{D}_2$ are division algebras having the same maximal
subfields. Then if one of the $\check{\mathcal{D}}_i$'s is a
division algebra then both of them are, in which case they have the
same maximal subfields.
\end{lemma}
\begin{proof}
Let $q_{\mathcal{D}_i}$ and $q_{\check{\mathcal{D}}_i}$ be the
corresponding quadratic forms. First, we will show that if a nonzero
$\bar{d} \in \bar{\mathcal{K}}^{(v)}$ is represented by
$q_{\check{\mathcal{D}}_1}$ then it is also represented by
$q_{\check{\mathcal{D}}_2}.$ Indeed, let $\bar{s}_1, \bar{t}_1,
\bar{u}_1 \in \bar{\mathcal{K}}^{(v)}$ be such that
$$q_{\check{\mathcal{D}}_1}(\bar{s}_1, \bar{t}_1, \bar{u}_1) =
\bar{d}.$$ Pick arbitrary lifts $s_1, t_1, u_1 \in
\mathcal{O}_{\mathcal{K} , v}$ and set $d = q_{\mathcal{D}_1}(s_1,
t_1, u_1).$ Since $\mathcal{D}_1$ and $\mathcal{D}_2$ have the same
maximal subfields, there exist $s_2, t_2, u_2 \in \mathcal{K}$ such
that $q_{\mathcal{D}_2}(s_2, t_2, u_2) = d.$ If $$\alpha :=
\min\{v(s_2) , v(t_2) , v(u_2)\} \geqslant 0$$ then taking
reductions we obtain
$$
q_{\check{\mathcal{D}}_2}(\bar{s}_2, \bar{t}_2, \bar{u}_2) =
\bar{d},
$$
as required. On the other hand, if $\alpha < 0$ then for
$$
s'_2 = \pi^{-\alpha} s_2, \ \ t'_2 = \pi^{-\alpha} t_2 \ \
\text{and} \ \ u'_2 = \pi^{-\alpha} u_2,
$$
where $\pi \in \mathcal{K}^{\times}$ is a uniformizer, we will have
$(\bar{s}'_2 , \bar{t}'_2 , \bar{u}'_2) \neq (\bar{0} , \bar{0} ,
\bar{0})$ and
$$
q_{\check{\mathcal{D}}_2}(\bar{s}'_2 , \bar{t}'_2 , \bar{u}'_2) =
\bar{0},
$$
i.e. $q_{\check{\mathcal{D}}_2}$ represents zero. But then, being
nondegenerate, it represents all elements of
$\bar{\mathcal{K}}^{(v)}.$

Now, if $\check{\mathcal{D}}_1$ is not a division algebra, then
$q_{\check{\mathcal{D}}_1}$ represents a nonzero square in
$\bar{\mathcal{D}}^{(v)}.$ By the above remark, the same is true for
$q_{\check{\mathcal{D}}_2},$ hence $\check{\mathcal{D}}_2$ is not a
division algebra, proving our first assertion. On the other hand, if
both $\check{\mathcal{D}}_1$ and $\check{\mathcal{D}}_2$ are
division algebras then by symmetry the above remark implies that
$q_{\check{\mathcal{D}}_1}$ and $q_{\check{\mathcal{D}}_2}$
represent the same elements, and therefore $\check{\mathcal{D}}_1$
and $\check{\mathcal{D}}_2$ have the same maximal subfields.
\end{proof}

\vskip.5mm

Another ingredient we need is a consequence of  the existence of the
specialization map in Milnor's $K$-theory. For a field $F,$ we let
$K_2(F)$ denote its second Milnor $K$-group, and for $a , b \in
F^{\times}$ let $\{a , b \} \in K_2(F)$ denote the corresponding
symbol. According to the Merkurjev-Suslin Theorem (cf. \cite{GS},
Ch. 8), for any field of characteristic $\neq 2$ there is an
isomorphism $K_2(F)/2K_2(F) \simeq \Br(F)_2$ sending $\{a , b\}$ to
the class of the quaternion algebra $\displaystyle \left(\frac{a ,
b}{F} \right).$ Let now $\mathcal{K}$ be a field complete with
respect to a discrete valuation $v.$ Then there exists a
homomorphism $s_v \colon K_2(\mathcal{K}) \to
K_2(\bar{\mathcal{K}}^{(v)})$ such that for any $a , b \in
\mathcal{K}^{\times}$ with $v(a) = v(b) = 0$ we have
$$
s_v(\{a , b\}) = \{\bar{a} , \bar{b}\}
$$
where $\bar{a} , \bar{b}$ are the images of $a , b$ in
$\bar{\mathcal{K}}^{(v)}$ (cf. \cite{GS}, Proposition 7.1.4); notice
that $s_v$ depends on the choice of a uniformizer in $\mathcal{K}.$
Combining $s_v$ with the Merkurjev-Suslin isomorphism, we obtain the
following.
\begin{lemma}\label{L:A-101}
Let $\mathcal{K}$ be a field complete with respect to a discrete
valuation $v.$ Assume that $\mathrm{char}\: \bar{\mathcal{K}}^{(v)}
\neq 2.$ Then there exists a homomorphism, called the {\rm
specialization homomorphism,} $\sigma_v \colon \Br(\mathcal{K})_2
\to \Br(\bar{\mathcal{K}}^{(v)})_2$ such that for any $a , b \in
\mathcal{K}^{\times}$ with $v(a) = v(b) = 0$ we have
$$
\sigma_v\left(\left[\left(\frac{a , b}{\mathcal{K}}
\right)\right]\right) = \left[\left(\frac{\bar{a} ,
\bar{b}}{\bar{\mathcal{K}}^{(v)}}\right)\right].
$$
\end{lemma}

\vskip1mm

We can extend the notion of the specialization homomorphism
$\sigma_v$ to any field $F$ with a discrete valuation $v$ such that
$\mathrm{char}\: \bar{F}^{(v)} \neq 2$ by defining it to be the
composition
$$
\Br(F)_2 \longrightarrow \Br(F_v)_2 \longrightarrow
\Br(\bar{F}^{(v)})_2
$$
of the extension of scalars with the specialization  homomorphism
described in Lemma \ref{L:A-101}. We then have
\begin{cor}\label{C:2}
Let $D$ be a quaternion algebra over $K.$ Then for any valuation $v$
of $K(x)$ we have $\sigma_v\left([D \otimes_K K(x)]\right) = [D
\otimes_K \overline{K(x)}^{(v)}].$
\end{cor}

\vskip2mm

{\it The conclusion of the proof of Theorem A.} It follows from
(AHBN) that (*) holds true for global fields (cf. the proof of
Corollary \ref{C:Fin-1}), so we may assume that $K$ is infinite.
Write the given central quaternion division algebras $D_i$ over
$K(x)$ in the form
$$
D_i = \left(\frac{a_i(x) , b_i(x)}{K(x)}  \right) \ \ \text{where} \
\ a_i(x), b_i(x) \in K[x] \ \text{for} \ i = 1, 2.
$$
Since $K$ is infinite, we can replace $x$ by $x - \alpha$ to ensure
that $a_i(0) , b_i(0) \neq 0$ for $i = 1, 2.$ We then set
$$
\check{D}_i = \left(\frac{a_i(0) , b_i(0)}{K} \right).
$$
By Corollary \ref{C:1}, we have
\begin{equation}\label{E:A-500}
[D_1][D_2] = [D \otimes_K K(x)]
\end{equation}
for some central quaternion algebra $D$ over $K,$ and all we need to
show is that the class $[D] \in \Br(K)$ is trivial. Let $v$ be the
valuation of $K(x)$ associated with $x;$ then $\overline{K(x)}^{(v)}
= K.$ It follows from Lemma \ref{L:A-101} and Corollary \ref{C:2}
that for the corresponding specialization homomorphism $\sigma_v
\colon \Br(K(x)) \to \Br(K)$ we have
$$
\sigma_v([D_i]) = [\check{D}_i] \ \ \text{for} \ i = 1, 2,  \
\text{and} \ \sigma_v([D \otimes_K K(x)]) = [D].
$$
Thus, it follows from (\ref{E:A-500}) that
\begin{equation}\label{E:A-501}
[D] = [\check{D}_1][\check{D}_2].
\end{equation}
However,
\begin{equation}\label{E:A-502}
[\check{D}_1] = [\check{D}_2].
\end{equation}
Indeed, if one of the classes $[\check{D}_i]$ is trivial then, by
Lemma \ref{L:A-201}, so is the other, and (\ref{E:A-502}) is
obvious. On the other hand, if both classes $[\check{D}_1]$ and
$[\check{D}_2]$ are nontrivial, then by Lemma \ref{L:A-201} the
division algebras $\check{D}_1$ and $\check{D}_2$ have the same
maximal subfields, and (\ref{E:A-502}) follows from our assumption
that (*) holds for $K.$ Now, (\ref{E:A-501}) and (\ref{E:A-502})
imply that $[D]$ is trivial, as required. \hfill $\Box$

\vskip3mm

 The specialization technique based on Lemmas \ref{L:A-201} and
\ref{L:A-101} can be used to analyze (*) for the fields of rational
functions on other curves,
%with infinitely many rational points,
which will be done elsewhere. In fact, it can also be used to obtain
some finiteness results pertaining to (*). More precisely, let us
consider the following property of a field $K:$

\vskip2mm

\noindent $(\Phi)$ \ \parbox[t]{15cm}{There exists $n = n(K)$ such
that for any central quaternion division algebra $D$ over $K,$
the~set of isomorphism classes of  central quaternion division
algebras over $K$ having the~same maximal subfields as $D$ contains
$\leqslant n$ elements.}

\vskip2mm

\begin{thm}\label{T:Fin-1}
Let $X$ be an absolutely irreducible smooth projective curve over
%an infinite finitely generated
a field $K$ of characteristic $\neq 2.$ Assume that the quotient
$\Br_{ur}(\mathcal{K})_2/\iota(\Br(K)_2),$ where
$\Br_{ur}(\mathcal{K})$ is the unramified Brauer group of
$\mathcal{K} = K(X)$ (over $K$) and $\iota \colon \Br(K) \to
\Br(\mathcal{K})$ is the natural map, is finite of order $m,$ and
there exists a family $\cL = \{ L \}$ of  odd degree extensions
$L/K$ such that

\vskip2mm

 $(\alpha)$ if $L \in \cL$ and $K \subset P \subset L$ then $P \in
\cL;$

\vskip2mm

 $(\beta)$ $\displaystyle \bigcup_{L \in \cL} X(L)$ is infinite;

\vskip2mm

 $(\gamma)$ \parbox[t]{15cm}{\baselineskip=3mm each $L \in \cL$ has property $(\Phi)$ and
 $\displaystyle \sup_{L \in \cL} n(L) =: n_0 < \infty$ where $n(L)$
is the number from the definition of  $(\Phi).$}

\vskip2mm

\noindent Then $\mathcal{K}$ has property $(\Phi)$ with
$n(\mathcal{K}) = m \cdot n_0.$ In particular, if $m < \infty,$ $K$
has property $(\Phi)$ and $X$ has infinitely many $K$-rational
points then $\mathcal{K}$ has property $(\Phi)$ with $n(\mathcal{K})
= m \cdot n(K).$
\end{thm}
\begin{proof}
Let $V^{\mathcal{K}}$ be the set of discrete valuations of
$\mathcal{K}$ trivial on $K,$ and for $v \in V^{\mathcal{K}}$ let
$\phi_v \colon \Br(\mathcal{K})_2 \to \mathrm{Hom}(G(v) , \mu_2),$
where $G(v)$ is the absolute Galois group of the residue field
$\bar{\mathcal{K}}^{(v)}$ and $\mu_2 = \{ \pm 1 \},$ denote the
composition $\rho_v \circ \lambda_v$ of the natural map $\lambda_v
\colon \Br(\mathcal{K}) \to \Br(\mathcal{K}_v)$ with the reduction
map $\rho_v \colon \Br(\mathcal{K}_v)_2 \to \mathrm{Hom}(G(v) ,
\mu_2).$ Then by definition $\Br_{ur}(\mathcal{K}) = \bigcap_{v \in
V^{\mathcal{K}}} \mathrm{Ker}\: \phi_v.$

Now, fix a central quaternion division algebra $D$ over
$\mathcal{K},$ and let $\mathcal{I}(D)$ denote the collection of
classes $[D'] \in \Br(\mathcal{K})$ for $D'$ a central quaternion
division algebra over $\mathcal{K}$ having the same maximal
subfields as $D.$ It follows from Lemmas \ref{L:P-1} and \ref{L:P-3}
that given $[D'] \in \mathcal{I}(D),$ for any $v \in
V^{\mathcal{K}}$ we have
\begin{equation}\label{E:Fin-0}
\rho_v([D' \otimes_{\mathcal{K}} \mathcal{K}_v]) = \rho_v([D
\otimes_{\mathcal{K}} \mathcal{K}_v]),
\end{equation}
i.e. $\phi_v([D']) = \phi_v([D])$ (cf. the proof of Lemma
\ref{L:1}). Thus, $\mathcal{I}(D) \subset [D] \cdot
\Br_{ur}(\mathcal{K})_2.$ By our assumption,
$\Br_{ur}(\mathcal{K})_2$ is the union of $m$ cosets $\mathcal{C}_1,
\ldots , \mathcal{C}_m$ modulo $\iota(\Br(K)_2).$ So, to prove that
$n(\mathcal{K}) = m \cdot n_0$ satisfies the definition of property
$(\Phi),$ it is enough to show that
\begin{equation}\label{E:Fin-5}
\vert \mathcal{I}(D) \cap \mathcal{C}_i \vert \leqslant n_0 \ \
\text{for all} \ i = 1, \ldots , m.
\end{equation}
If $\mathcal{I}(D) \cap \mathcal{C}_i = \emptyset$ then there is
nothing to prove; otherwise, all central quaternion algebras $D'$
with $[D'] \in \mathcal{I}(D) \cap \mathcal{C}_i$ have the same
maximal subfields as $D.$ So, it is enough to show that for any
central quaternion division algebra $D$ over $\mathcal{K},$ the
number of classes $[D'],$ where $D'$ is a central quaternion
division algebra having the same maximal subfields as $D$ and such
that $[D'] \in [D] \cdot \iota(\Br(K)_2),$ is $\leqslant n_0.$
\begin{lemma}\label{L:Fin-5}
For a central division algebra $\Delta$ over $K$ of degree $\ell =
2^d,$ the algebra $\Delta \otimes_K \mathcal{K}$ is also a division
algebra. Consequently, {\rm (1)} if for such $\Delta$ the algebra
$\Delta \otimes_K \mathcal{K}$ is Brauer-equivalent to a quaternion
algebra then $\Delta$ is itself a quaternion algebra, and {\rm (2)}
the natural map $\Br(K)_2 \to \Br(\mathcal{K})_2$ is injective.
\end{lemma}
\begin{proof}
By $(\beta),$ there is an odd degree extension $L/K$ and a rational
point $p_0 \in X(L).$ Let $v_0 \in V^{\mathcal{K}}$ be the valuation
of $\mathcal{K}$ obtained as the restriction of the valuation of
$L(X)$ associated with $p_0.$ Then the residue field $P =
\bar{\mathcal{K}}^{(v_0)}$ is an odd degree extension of $K.$ Let
$f(x_1, \ldots , x_{\ell^2})$ be the homogeneous polynomial of
degree $\ell$ representing the reduced norm
$\mathrm{Nrd}_{\Delta/K}.$ If $\Delta \otimes_K \mathcal{K}$ is not
a division algebra then $f$ represents zero over $\mathcal{K}.$ Then
$f$ also represents zero over $P,$ i.e. $\Delta \otimes_K P$ is not
a division algebra. This, however, cannot happen as $\ell = 2^d$ and
$[P : K]$ is odd (cf. \cite{Pi}, \S 13.4, part (vi) of the
proposition). A contradiction, proving our first claim. The
remaining assertions easily follow.
\end{proof}

For central quaternion division algebras $D$ and $D'$ over
$\mathcal{K}$ having the same maximal subfields, we have $D \otimes
D' \simeq M_2(D'')$ for some central quaternion algebra $D''$ (cf.
\cite{GS}, Lemma 1.5.2). If in addition $[D'] = [D][\Delta \otimes_K
\mathcal{K}]$ for a central division algebra $\Delta$ over $K$ with
$[\Delta] \in \Br(K)_2$ then it follows from the lemma that either
$\Delta = K$ or $\Delta$ is a quaternion division algebra. So, if we
let $\mathcal{J}(D)$ denote the set of classes $[\Delta] \in
\Br(K)_2$ where $\Delta$ is a central quaternion algebra over $K$
such that the class $[D][\Delta \otimes_K \mathcal{K}] \in
\Br(\mathcal{K})$ is represented by a central quaternion division
algebra $D'$ over $\mathcal{K}$ having the same maximal subfields as
$D,$ then to prove (\ref{E:Fin-5}) it is enough to show that $\vert
\mathcal{J}(D) \vert \leqslant n_0$ for any $D.$ Finally, since for
an odd degree extension $P/K$ the natural map $\lambda_P \colon
\Br(K)_2 \to \Br(P)_2$ is injective, it is enough to find, for a
fixed $D,$ such an extension for which $\vert
\lambda_P(\mathcal{J}(D)) \vert \leqslant n_0,$ and this is what we
are going to do now.

\vskip1mm

Fix a central quaternion division algebra $\displaystyle D = \left(
\frac{a , b}{\mathcal{K}} \right)$ over $\mathcal{K}.$ It follows
from assumption $(\beta)$ in the statement of the theorem that there
exists $L \in \cL$   and a point $p_0 \in X(L)$ which is neither a
zero nor a pole of $a$ or $b.$ As in the proof of Lemma
\ref{L:Fin-5}, we let $v_0$ be the restriction to $\mathcal{K}$ of
the valuation of $L(X)$ associated with $p_0.$ Then
\begin{equation}\label{E:Fin-1}
v_0(a) = v_0(b) = 0,
\end{equation}
and the residue field $P = \bar{\mathcal{K}}^{(v_0)} \subset L$
belongs to $\cL$ (by $(\alpha)$); in particular $[P : K]$ is odd.
Any other central quaternion division algebra $D'$ over
$\mathcal{K}$ having the same maximal subfields as $D$ can be
written in the form $\displaystyle D' = \left( \frac{a ,
b'}{\mathcal{K}} \right)$ for some $b' \in \mathcal{K}^{\times}.$
Furthermore, it follows from (\ref{E:Fin-1}) that $D$ is unramified
at $v_0,$ and then due to (\ref{E:Fin-0}), $D'$ is also unramified
at $v_0.$ Then we can choose $b'$ so that $v_0(b') = 0$ (cf.
\cite{GaSa}, 3.4).

Consider the quaternion algebras $\displaystyle \check{D} = \left(
\frac{\bar{a} , \bar{b}}{P} \right)$ and $\displaystyle \check{D}' =
\left( \frac{\bar{a} , \bar{b}'}{P} \right)$ where $\bar{a},$
$\bar{b}$ and $\bar{b}'$ are the images in $P =
\bar{\mathcal{K}}^{(v_0)}$ of $a,$ $b$ and $b',$ respectively. Set
$\check{\mathcal{I}}(\check{D}) = \{ e \}$ if $\check{D} \simeq
M_2(P),$ and let $\check{\mathcal{I}}(\check{D})$ denote the set of
classes $[\Delta] \in \Br(P)$ where $\Delta$ is a central quaternion
division algebra over $P$ having the same maximal subfields as
$\check{D}$ if the latter is a division algebra. Since $P \in \cL,$
it follows from $(\gamma)$ that $\vert
\check{\mathcal{I}}(\check{D}) \vert \leqslant n_0.$ On the other
hand, according to Lemma \ref{L:A-201} we have  $[\check{D}'] \in
\check{\mathcal{I}}(\check{D}).$ Now, if $[D'] = [D] [\Delta
\otimes_K \mathcal{K}]$ where $[\Delta] \in \mathcal{J}(D)$ then it
follows from Lemma \ref{L:A-101} that
$$
[\check{D}'] = [\check{D}] [\Delta \otimes_K P] \ \ \text{in} \ \
\Br(P).
$$
This shows that $\lambda_P(\mathcal{J}(D)) \subset [\check{D}]^{-1}
\check{\mathcal{I}}(\check{D}),$ and consequently,
$$
\vert \lambda_P(\mathcal{J}(D)) \vert \leqslant \vert
\check{\mathcal{I}}(\check{D}) \vert \leqslant n_0,
$$
as required.
\end{proof}

\vskip1mm

\begin{cor}\label{C:Fin-1}
Let $\mathscr{H}$ be a finitely generated subgroup of the absolute
Galois group $\Ga(\overline{\Q}/\Q),$ and let $K =
\overline{\Q}^{\mathscr{H}}$ be the corresponding fixed field.
Furthermore, let $F(x , y)$ be an absolutely irreducible polynomial
over $K$ such that at least one of the numbers $\deg F,$ $\deg_x F$
or $\deg_y F$ is odd. Then the field $\mathcal{K} = K(X_0)$ of
$K$-rational functions on the affine curve $X_0$ given by $F(x , y)
= 0$ has property~$(\Phi).$
\end{cor}
\begin{proof}
We first note the following elementary statement.
\begin{lemma}\label{L:Fin-1}
Let $F(x , y)$ be an absolutely irreducible polynomial over an
arbitrary field $K$ such that one of the numbers $\deg F,$ $\deg_x
F$ or $\deg_y F$ is odd. Let $X$ be a smooth projective $K$-defined
model for the affine curve $X_0$ given by $F(x , y) = 0.$ Then
$\displaystyle \bigcup_{L \in \cL} X(L),$ where $\cL$ is the family
of all finite extensions $L/K$ of odd degree, is infinite.
\end{lemma}
\begin{proof}
It is enough to show that $\fX_{\cL} = \bigcup_{L \in \cL} X_0(L)$
is infinite. Assume the contrary, i.e. $\fX_{\cL} = \{ (x_1 , y_1),
\ldots , (x_r , y_r) \}.$ We will first consider the case where $d
:=\deg_x F$ is odd. We have $F(x , y) = f_d(y) x^d + f_{d-1}(y)
x^{d-1} + \cdots $ with $f_d \not\equiv 0,$ so one can find an odd
degree extension $K'/K$ containing an element $y_0 \notin \{y_1,
\ldots , y_r\}$ such that $f_d(y_0) \neq 0$ (if $K$ is infinite then
such an element $y_0$ can already be found in $K' = K$). Then the
degree of $\varphi(x) = F(x , y_0) \in K'[x]$ is $d,$ hence odd, and
therefore $\varphi(x)$ has an irreducible factor $\psi(x)$ of odd
degree. Let $x_0$ be a root of $\psi(x)$ (in a fixed algebraic
closure of $K$), and set $L = K'(x_0).$ Then $L$ is of odd degree
over $K',$ hence over $K,$ i.e. $L \in \cL.$ So, $(x_0 , y_0) \in
X_0(L) \subset \fX_{\cL},$ contradicting our construction. The case
where $\deg_y F$ is odd is reduced to the case just considered by
switching $x$ and $y.$ Finally, if $\deg F$ is odd then one can find
an odd degree extension $K''/K$ and $a \in K''$ so that for $\Phi(x
, y) = F(x , y + a x)$ we have $\deg_x \Phi = \deg F,$ hence odd
(again, if $K$ is infinite one can find such an $a$ already in $K''
= K$). Then our claim holds for the $K''$-defined curve given by
$\Phi(x , y) = 0,$ which implies its truth for $X_0$ as any odd
degree extension $L/K''$ is of odd degree over $K.$
\end{proof}

Next, we recall that as follows from (AHBN) (\cite{Pi}, \S 18.5)
{\it any} algebraic extension $L/\Q$  satisfies (*), i.e. it
satisfies $(\Phi)$ with $n(L) = 1.$ Indeed, let $D_1$ and $D_2$ be
central quaternion division algebras  over $L$ having the same
maximal subfields. Assume that $D_1 \not\simeq D_2;$ then $D := D_1
\otimes_L D_2$ represents a nontrivial class in $\Br(L).$ We can
find a {\it finite} extension $L^0/\Q$ contained in $L$ and central
quaternion division algebras $D_i^0$ over $L_0$ such that $D_i =
D_i^0 \otimes_{L_0} L$ for $i = 1, 2.$ Set $D^0 = D_1^0
\otimes_{L^0} D_2^0.$ Then for any finite extension $P/L^0,$
contained in $L,$ the algebra $D^0_P = D^0 \otimes_{L^0} P$
represents a nontrivial class in $\Br(P),$ and therefore by (AHBN)
there exists a valuation $w$ of $P$ (which can be archimedean) such
that the class $[D^0_P \otimes_P P_{w}] \in \Br(P_{w})$ is
nontrivial. The standard argument using the nonemptiness of the
inverse limit of an inverse system of nonempty finite sets shows
that there exists a valuation $\tilde{v}$ of $L$ such that for any
finite subextension $L^0 \subset P \subset L$ and $v_P := \tilde{v}
\vert P,$ the class $[D^0_P \otimes_P P_{v_P}] \in \Br(P_{v_P})$ is
nontrivial. Let $v_0 = \tilde{v} \vert L^0.$ Since
$\Br(L^0_{v_0})_2$ is of order $\leqslant 2$ and the class $[D^0
\otimes_{L^0} L^0_{v_0}]$ is notrivial, one of the classes $[D_i^0
\otimes_{L^0} L^0_{v_0}],$ where $i = 1, 2,$ is trivial and the
other is nontrivial. Suppose that $[D_1^0 \otimes_{L^0} L^0_{v_0}]$
is trivial. Then for any finite subextension $L^0 \subset P \subset
L,$ the class $[D_2^0 \otimes_{L^0} P_{v_P}] \in \Br(P_{v_P})$ is
nontrivial. Let $V$ be the (finite) set of ramification places of
$D_1^0.$ Then $v_0 \notin V,$ so by the weak approximation theorem
there exists $t \in (L^0)^{\times}$ such that $t \notin
{(L^0_v)^{\times}}^{2}$ for all $v \in V$ and $t \in
{(L^0_{v_0})^{\times}}^{2}.$ It follows from (AHBN) that
$L^0(\sqrt{t})$ is isomorphic to a maximal subfield of $D^0_1$ (cf.
\cite{Pi}, Cor. b in \S 18.4), hence $L(\sqrt{t})$ is isomorphic to
a maximal subfield of $D_1.$ Since $D_1$ and $D_2$ have the same
maximal subfields, there exists a finite subextension $L^0 \subset P
\subset L$ such that $P(\sqrt{t})$ is isomorphic to a maximal
subfield of $D_2^0 \otimes_{L^0} P.$ However by our construction $t
\in {P_{v_P}^{\times}}^{2},$ so the latter is impossible as $D_2^0
\otimes_{L^0} P_{v_P}$ is a division algebra.

\vskip1mm

Finally, since $\Ga(\overline{K}/K) = \mathscr{H}$ is finitely
generated, the field $K$ is of type (F) as defined by Serre
(\cite{Se-CG}, Ch. III, \S 4.2), and therefore $H^1(K , C)$ is
finite for any finite $\Ga(\overline{K}/K)$-module $C$ ({\it loc.
cit.}, Theorem 4). Then it follows from exact sequence (9.25) in
\cite{GMS}, p. 27, that for $\mathcal{K} = K(X_0) = K(X),$ where $X$
is the $K$-defined smooth projective model for $X_0,$ the quotient
$\Br_{ur}(\mathcal{K})_2/\iota(\Br(K)_2)$ is finite. Thus, our claim
follows from Theorem \ref{T:Fin-1} applied to the family $\cL = \{ L
\}$ of all odd degree extensions $L/K.$
\end{proof}

\vskip1.5mm

{\bf Remarks 4.10.} 1. Lemma \ref{L:Fin-1} (and hence Corollary
\ref{C:Fin-1}) applies to any elliptic curve as well as to any
hyperelliptic curve given by $y^2 = f(x)$ where $f$ is a polynomial
of odd degree without multiple roots.

\vskip1mm

2. We observe that (*) for a field $K$ is equivalent to $(\Phi)$
with $n(K) = 1.$ For $X = \mathbb{P}^1_K,$ Faddeev's exact sequence
yields $m = 1,$ so we obtain from Theorem \ref{T:Fin-1} (assuming,
as we may,  $K$ to be infinite) that $n(K) = 1$ implies $n(K(x)) =
1,$ which is precisely Theorem A. Thus, Theorem \ref{T:Fin-1}
contains Theorem A as a particular case. For the clarity of
exposition, however, we decided to give first a streamlined proof of
Theorem A which is not loaded with extra technical details.

\vskip2mm

\noindent {\it Acknowledgements.} We are grateful to Louis Rowen and
David Saltman for offering their comments on and corrections to an
early version of this paper.

\vskip5mm

\bibliographystyle{amsplain}

\end{document}